\documentclass[a4paper,12pt]{article}
\usepackage[top=2.5cm,bottom=2.5cm,left=2.5cm,right=2.5cm]{geometry}
\usepackage{cite,amsmath,amssymb}
    \usepackage[margin=1cm,%
                font=small,%
                format=hang,%
                labelsep=period,%
                labelfont=bf]{caption}
\usepackage[algoruled,linesnumbered]{algorithm2e}
\usepackage{algorithmic}

\usepackage{amsfonts,epsf,here,graphicx}
\usepackage{latexsym,multirow,rotating}
\usepackage{pgf,tikz,subfigure}
\usepackage{epstopdf}

\newtheorem{theorem}{\bf Theorem}[section]

\newtheorem{lemma}[theorem]{\bf Lemma}

\newtheorem{definition}[theorem]{\bf Definition}

\newcommand{\qed}{\hfill $\square$ \bigskip}

\newcommand{\aut}{\rm Aut}

\begin{document}

\baselineskip=0.30in
\vspace*{40mm}

\begin{center}
{\LARGE \bf The Graovac-Pisanski Index of Zig-Zag Tubulenes and the Generalized Cut Method}
\bigskip \bigskip

{\large \bf Niko Tratnik
}
\bigskip\bigskip

\baselineskip=0.20in

\textit{Faculty of Natural Sciences and Mathematics, University of Maribor, Slovenia} \\
{\tt niko.tratnik@um.si}
\medskip

\bigskip\medskip

(\today)

\end{center}

\noindent
\begin{center} {\bf Abstract} \end{center}

\vspace{3mm}\noindent
The Graovac-Pisanski index, which is  also called the modified Wiener index, was introduced in 1991 by A. Graovac and T. Pisanski.
This variation of the classical Wiener index takes into account the symmetries of a graph. In 2016 M. Ghorbani and S. Klav\v zar calculated this index by using the cut method, which we generalize in this paper. Moreover, we prove that in some cases the automorphism group of a zig-zag tubulene is isomorphic to the direct product of a dihedral group and a cyclic group. Finally, the closed formulas for the Graovac-Pisanski index of zig-zag tubulenes are calculated. 
%\vspace{5mm}

\baselineskip=0.30in

\noindent {\bf Key words:} Modified Wiener index; Graovac-Pisanski index; Zig-zag tubulene; Automorphism group; Cut method

% \medskip\noindent
% {\bf AMS Subj. Class:} 92E10, 05C12

%%%%%%%%%%%%%%%%%%%%%%%%%%%%%%%%%%%%%%%%%%%%%%%%%%%%%%%%%%%%%%%%%%%%%
%%%%%%%%%%%%%%%%%%%%%%%%%%%%%%%%%%%%%%%%%%%%%%%%%%%%%%%%%%%%%%%%%%%%%
\section{Introduction}
%%%%%%%%%%%%%%%%%%%%%%%%%%%%%%%%%%%%%%%%%%%%%%%%%%%%%%%%%%%%%%%%%%%%%
%%%%%%%%%%%%%%%%%%%%%%%%%%%%%%%%%%%%%%%%%%%%%%%%%%%%%%%%%%%%%%%%%%%%%
The Graovac-Pisanski index was introduced by A. Graovac and T. Pisanski in 1991 \cite{graovac} under the name modified Wiener index.  Unfortunately, the same name was later used for   different variations of the Wiener index \cite{ni-tr,gu-vu,li-li}. 
As sugested by M. Ghorbani and S. Klav\v zar in \cite{ghorbani}, we use the name Graovac-Pisanski index. 
In \cite{graovac} they applied the symmetry group of a graph to obtain an algebraic modification of the classical Wiener index, so
beside distances in a graph this index also considers its symmetries. The automorphisms (symmetries) of a molecular graph represent the isomers of a corresponding molecule. 

It was shown in \cite{ashrafi_koo_diu1} that the quotient of the Wiener index and the Graovac-Pisanski index is strongly correlated with the topological efficiency for some fullerene molecules. The topological efficiency was introduced in \cite{cataldo,ori} as a tool for the classification of the stability of molecules. It was also pointed out that electronic properties of carbon systems are deeply connected to the topology of their graphs (see book \cite{basak}, pp. 3--21).

For a vertex-transitive graph the Graovac-Pisanski index coincides with the Wiener index.
The previous work on the symmetries of different nanostructures can be found in \cite{ashrafi_diu,ashrafi_koo_diu,damn,ghorbani_hak}
and  particularly for the Graovac-Pisanski index in \cite{ashrafi_sha,koo_ashrafi,koo_ashrafi2,sha_ashrafi}. In addition, the cut method for this index  was developed in \cite{ghorbani}, where it was proved that the computation of the Graovac-Pisanski index
can be reduced to the computation of the Wiener indices of the appropriately weighted
quotient graphs.

In the present paper we first generalize the cut method so that  it is valid for any partition of the edge set which is coarser than $\Theta^*$-partition. Next, the automorphisms for an important family of chemical graphs, the zig-zag tubulenes, are described and finally, the closed formulas for their Graovac-Pisanski index are calculated.

\section{Preliminaries}

Unless stated otherwise, the graphs considered in this paper are finite and connected. The {\em distance} $d_G(x,y)$ between vertices $x$ and $y$ of a graph $G$ is the length of a shortest path between vertices $u$ and $v$ in $G$. We also write $d(x,y)$ for $d_G(x,y)$.
\bigskip

\noindent
The {\em Wiener index} of a graph $G$ is defined as $\displaystyle{W(G) = \frac{1}{2} \sum_{u \in V(G)} \sum_{v \in V(G)} d_G(u,v)}$. Moreover, if $S \subseteq V(G)$, then $W(S) = \frac{1}{2} \sum_{u \in S} \sum_{v \in S} d_G(u,v)$.
\bigskip

\noindent
Now we extend the above definition to weighted graphs as follows. Let $G$ be a connected graph and let $w:V(G)\rightarrow {\mathbb R}^+$ be a given function. Then $(G,w)$ is a {\em vertex-weighted graph}. The \textit{vertex-weighted Wiener index} of $(G,w)$ is defined as
$$W(G,w) = \frac{1}{2} \sum_{u \in V(G)} \sum_{v \in V(G)} w(u)w(v)d_G(u,v).$$

\noindent  
An \textit{isomorphism of graphs} $G$ and $H$ with $|E(G)|=|E(H)|$ is a bijection $f$ between the vertex sets of $G$ and $H$, $f: V(G)\to V(H)$,
such that for any two vertices $u$ and $v$ of $G$ it holds that if $u$ and $v$ are adjacent in $G$ then $f(u)$ and $f(v)$ are adjacent in $H$. When $G$ and $H$ are the same graph, the function $f$ is called an \textit{automorphism} of $G$. The composition of two automorphisms is another automorphism, and the set of automorphisms of a given graph $G$, under the composition operation, forms a group ${\aut}(G)$, which is called the \textit{automorphism group} of the graph $G$.
\bigskip

\noindent 
The Graovac-Pisanski index of a graph $G$, $\widehat{W}(G)$, is defined as
$$\widehat{W}(G) = \frac{|V(G)|}{2 |{\aut}(G)|} \sum_{u \in V(G)} \sum_{\alpha \in {\aut}(G)} d_G(u, \alpha(u)).$$
Roughly speaking, the Graovac-Pisanski
index measures how far the vertices of a graph are moved on the average by its automorphisms.
\bigskip

\noindent Next, we repeat some important concepts from group theory. If $G$ is a group and $X$ is a set, then a \textit{group action} $\phi$ of $G$ on $X$ is a function $\phi :G \times X \to X$
that satisfies the following: $\phi(e,x) = x$ for any $x \in X$ (here, $e$ is the neutral element of $G$) and $\phi(gh,x)=\phi(g,\phi(h,x))$ for all $g,h \in G$ and $x \in X$. The \textit{orbit} of an element $x$ in $X$ is the set of elements in $X$ to which $x$ can be moved by the elements of $G$, i.e. the set $\lbrace \phi(g,x) \, | \, g \in G \rbrace$. If $G$ is a graph and ${\aut}(G)$ the automorphism group, then $\phi: {\aut}(G) \times V(G) \to V(G)$, defined by $\phi(\alpha,u) = \alpha(u)$ for any $\alpha \in {\aut}(G)$, $u \in V(G)$, is called the \textit{natural action} of the group ${\aut}(G)$ on $V(G)$.

\noindent
It was shown in \cite{graovac} that if $V_1, \ldots, V_t$ are the orbits under the natural action of the group ${\aut}(G)$ on $V(G)$, then
\begin{equation}
\label{formula}
\widehat{W}(G) = |V(G)| \sum_{i=1}^t \frac{1}{|V_i|}W(V_i).
\end{equation}

\noindent We also introduce $W'(G) = \sum_{i=1}^t W(V_i)$, which is the sum of the Wiener indices of orbits of $G$.
\bigskip

\noindent 
The \textit{dihedral group} $D_n$ is the group of symmetries of a regular polygon with $n$ sides. Therefore, the group $D_n$ has $2n$ elements. The \textit{cyclic group} $\mathbb{Z}_n$ is a group that is generated by a single element of order $n$. Given groups $G$ and $H$, the \textit{direct product} $G \times H$ is defined as follows. The underlying set is the Cartesian product $G \times H$ and the binary operation on $G \times H$ is defined component-wise: $(g_1,h_1)(g_2,h_2)=(g_1g_2,h_1h_2)$, $(g_1,h_1),(g_2,h_2) \in G \times H$.
\bigskip

\noindent
If $G$ and $H$ are groups, then a \textit{group isomorphism} is a bijective function $f: G \rightarrow H$ such that for all $u$ and $v$ in $G$ it holds $f(uv)=f(u)f(v)$.
\bigskip

\noindent
Two edges $e_1 = u_1 v_1$ and $e_2 = u_2 v_2$ of graph $G$ are in relation $\Theta$, $e_1 \Theta e_2$, if
$$d_G(u_1,u_2) + d_G(v_1,v_2) \neq d_G(u_1,v_2) + d_G(u_1,v_2).$$
Note that this relation is also known as Djokovi\' c-Winkler relation.
The relation $\Theta$ is reflexive and symmetric, but not necessarily transitive.
We denote its transitive closure (i.e.\ the smallest transitive relation containing $\Theta$) by $\Theta^*$. Let $ \mathcal{E} = \lbrace E_1, \ldots, E_r \rbrace$ be the $\Theta^*$-partition of the set $E(G)$. Then we say that a partition $\lbrace F_1, \ldots, F_k \rbrace$ of $E(G)$ is \textit{coarser} than $\mathcal{E}$
if each set $F_i$ is the union of one or more $\Theta^*$-classes of $G$. The following lemma was proved in \cite{graham}.

\begin{lemma} \cite{graham} \label{pomozna}
Let $E$ be a $\Theta^*$-class of a connected graph $G$, and let $x,y \in V(G)$. If $P$ is a
shortest $x,y$-path and $Q$ an arbitrary $x,y$-path, then $|E(Q) \cap E| \geq |E(P) \cap E|$.
\end{lemma}

\noindent
Suppose $G$ is a graph and $F \subseteq E(G)$. The \textit{quotient graph} $G / F$ is a graph whose vertices are connected components of the graph $G - F$, such that two components $C_1$ and $C_2$ are adjacent in $G / F$ if some vertex in $C_1$ is
adjacent to a vertex of $C_2$ in $G$.
\bigskip

Finally, we will formally define open-ended carbon nanotubes, also called tubulenes (see \cite{sa}). Choose any lattice point in the hexagonal lattice as the origin $O$. Let $\overrightarrow{a_1}$ and $\overrightarrow{a_2}$ be the two basic lattice vectors.
 Choose a vector $ \overrightarrow{OA} =n\overrightarrow{a_1}+m \overrightarrow{a_2}$
such that $n$ and $m$ are two integers and $|n|+|m|>1$, $nm\neq -1$. Draw two straight lines $L_1$ and $L_2$ passing through
$O$ and $A$ perpendicular to $O A$, respectively. By rolling up the hexagonal strip between $L_1$ and $L_2$ and gluing $L_1$ and $L_2$ such
that $A$ and $O$ superimpose, we can obtain a hexagonal tessellation $\mathcal{HT}$ of the cylinder. $L_1$ and $L_2$ indicate the direction of
the axis of the cylinder. Using the terminology of graph theory, a {\em tubulene} $T$ is defined to be the finite graph induced by all
the hexagons of $\mathcal{HT}$ that lie between $c_1$ and $c_2$, where $c_1$ and $c_2$ are two vertex-disjoint cycles of $\mathcal{HT}$ encircling the axis of
the cylinder.  The vector $\overrightarrow{OA}$ is called the {\em chiral vector} of $T$ and  the cycles $c_1$ and $c_2$ are the two open-ends of $T$.

For any  tubulene $T$, if its chiral vector is $ n \overrightarrow{a_1} + m \overrightarrow{a_2}$, $T$ will be called an $(n,m)$-type tubulene, see Figure \ref{zig-zag}. If $T$ is a $(n,m)$-type tubulene where $n=0$ or $m=0$, we call it a \textit{zig-zag tubulene}.

\begin{figure}[!htb]
	\centering
		\includegraphics[scale=0.6, trim=0cm 0cm 1cm 0cm]{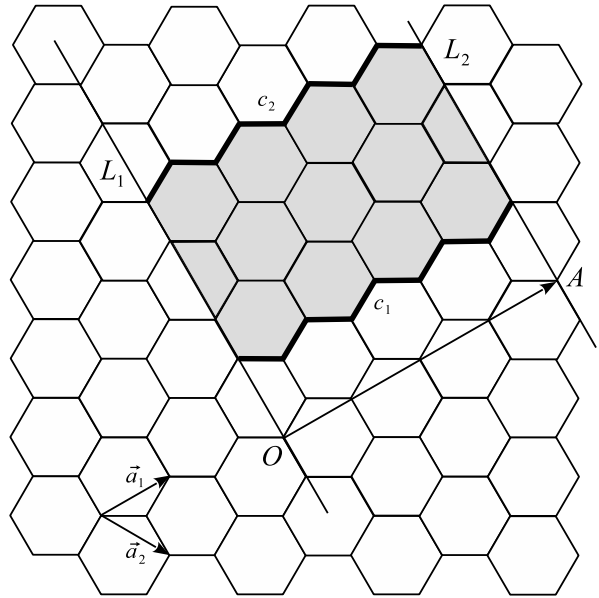}
\caption{Zig-zag tubulene $ZT(3,4)$, which is $(4,0)$-type tubulene.}
	\label{zig-zag}
\end{figure}

\section{The generalized cut method}
\label{sec:cut}
%%%%%%%%%%%%%%%%%%%%%%%%%%%%%%%%%%%%%%%%%%%%%%%%%%%%%%%%%%%%%%%%%%%%%
%%%%%%%%%%%%%%%%%%%%%%%%%%%%%%%%%%%%%%%%%%%%%%%%%%%%%%%%%%%%%%%%%%%%%

In this section the main result of paper \cite{ghorbani} is generalized such that it is valid for any partition of the edge set which is coarser than $\Theta^*$-partition. We start with the following definition.
\begin{definition}
Let $G$ be a connected graph and $\lbrace F_1, \ldots, F_k \rbrace$ a partition coarser than the $\Theta^*$-partition. For any $x \in V(G)$ and $j \in \lbrace 1, \ldots, k \rbrace$ we denote by $\ell_j(x)$ the connected component of the graph $G- F_j$ which contains $x$.
\end{definition}

The following lemma is crucial for the main result of this section. The ideas for the proof can be found inside the proof of Theorem 3.3 of \cite{nad_klav}. For the sake of completeness we give the proof anyway.

\begin{lemma} \cite{nad_klav} \label{distance}
Let $G$ be a connected graph. If $\lbrace F_1, \ldots, F_k \rbrace$ is a partition coarser than the $\Theta^*$-partition, then for any $x,y \in V(G)$ it holds
$$d_G(x,y) = \sum_{j=1}^k d_{G / F_j}(\ell_j(x),\ell_j(y)).$$
\end{lemma}

\begin{proof}
Let $x,y \in V(G)$ and let $P$ be a shortest path between $x$ and $y$. Since $\lbrace F_1, \ldots, F_k \rbrace$ is a partition of the edge set, we have
$$d_G(x,y) = |E(P)| = \sum_{j=1}^k |E(P) \cap F_j|.$$
Therefore, it suffices to show that $|E(P) \cap F_j| = d_{G / F_j}(\ell_j(x),\ell_j(y))$ for any $j \in \lbrace 1, \ldots, k \rbrace$.

\noindent
First suppose that $|E(P) \cap F_j| = r$. Obviously, a path of length $r$ can be constructed in $G / F_j$ between $\ell_j(x)$ and $\ell_j(y)$. Hence, $|E(P) \cap F_j| \geq d_{G / F_j}(\ell_j(x),\ell_j(y))$.

\noindent
Finally, suppose that $d_{G / F_j}(\ell_j(x),\ell_j(y)) = r$ and let $\ell_j(x) = C_0, C_1, \ldots, C_{r-1},$ $C_r = \ell_j(y)$ be a shortest path in $G / F_j$ between $\ell_j(x)$ and $\ell_j(y)$. Therefore, a path $Q$ from $x$ to $y$ can be obtained such that $|E(Q) \cap F_j| = r$. Suppose that $F_j$ is the union of $\Theta^*$-classes $E_1, \ldots, E_s$. Since it follows from Lemma \ref{pomozna} that $|E(Q) \cap E_i| \geq |E(P) \cap E_i|$ for any $i \in \lbrace 1, \ldots, s \rbrace$, we obtain $r = |E(Q) \cap F_j| \geq |E(P) \cap F_j|$ and therefore, we also have $|E(P) \cap F_j| \leq d_{G / F_j}(\ell_j(x),\ell_j(y))$. \qed
\end{proof}

\noindent Now everything is prepared for the final result of this section.

\begin{theorem} \label{gra_pis}
Let $G$ be a connected graph and let $V_1, \ldots, V_t$ be the orbits under the natural action of the group $\aut(G)$ on $V(G)$. If $\lbrace F_1, \ldots, F_k \rbrace$ is a partition coarser than the $\Theta^*$-partition, then
$$\widehat{W}(G) = |V(G)| \sum_{i=1}^t \frac{1}{|V_i|} \sum_{j=1}^k W(G / F_j, w_{ij}),$$
where $w_{ij}(C) = |V_i \cap C|$ for any $i \in \lbrace 1, \ldots, t \rbrace$, $j \in \lbrace 1, \ldots, k \rbrace$, and $C \in V(G / F_j)$.
\end{theorem}

\begin{proof}
From Equation \ref{formula} and Lemma \ref{distance} we obtain
\begin{eqnarray*}
\widehat{W}(G) & = & |V(G)| \sum_{i=1}^t \frac{1}{|V_i|} \sum_{\lbrace x,y \rbrace \subseteq V_i}d_G(x,y) \\
& = & |V(G)| \sum_{i=1}^t \frac{1}{|V_i|} \sum_{\lbrace x,y \rbrace \subseteq V_i} \Bigg( \sum_{j=1}^k d_{G / F_j}(\ell_j(x),\ell_j(y)) \Bigg) \\
& = & |V(G)| \sum_{i=1}^t \frac{1}{|V_i|} \Bigg( \sum_{\lbrace x,y \rbrace \subseteq V_i}  \sum_{j=1}^k d_{G / F_j}(\ell_j(x),\ell_j(y)) \Bigg) \\
& = & |V(G)| \sum_{i=1}^t \frac{1}{|V_i|} \Bigg( \sum_{j=1}^k  \sum_{\lbrace x,y \rbrace \subseteq V_i} d_{G / F_j}(\ell_j(x),\ell_j(y)) \Bigg) \\
& = & |V(G)| \sum_{i=1}^t \frac{1}{|V_i|}  \sum_{j=1}^k  \Bigg( \sum_{\lbrace x,y \rbrace \subseteq V_i} d_{G / F_j}(\ell_j(x),\ell_j(y)) \Bigg).
\end{eqnarray*}

\noindent
Obviously, for $C, D \in V(G / F_j), C \neq D$ it holds that the number of unordered pairs $\lbrace x, y \rbrace \subseteq V_i$ for which $x \in C$, $y \in D$ is exactly $w_{ij}(C)w_{ij}(D)$. Therefore, it follows

\begin{eqnarray*}
\widehat{W}(G) & = & |V(G)| \sum_{i=1}^t \frac{1}{|V_i|}  \sum_{j=1}^k  \Bigg( \sum_{\lbrace C,D \rbrace \subseteq V(G / F_j)} w_{ij}(C)w_{ij}(D)d_{G / F_j}(C,D) \Bigg) \\
& = & |V(G)| \sum_{i=1}^t \frac{1}{|V_i|} \sum_{j=1}^k W(G / F_j, w_{ij})
\end{eqnarray*}

\noindent and the proof is complete. \qed
\end{proof}

In the rest of this section we show with an example how Theorem \ref{gra_pis} can be used. Let $T$ be a tree from Figure \ref{primer_drevo}. 

\begin{figure}[H] 
\begin{center}
\includegraphics[scale=0.7]{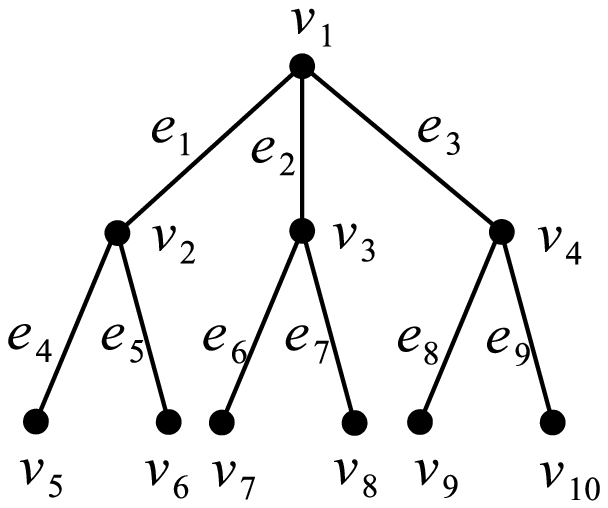}
\end{center}
\caption{\label{primer_drevo} Tree $T$.}
\end{figure}

It is easy to see that the natural action of the group ${\aut}(T)$ on $V(T)$ has three orbits: $V_1 = \lbrace v_1 \rbrace$, $V_2 = \lbrace v_2, v_3, v_4 \rbrace$, and $V_3 = \lbrace v_5, v_6, v_7, v_8, v_9, v_{10} \rbrace$. Since in a tree every $\Theta^*$-class is a single edge, the sets $F_1 = \lbrace e_1, e_2, e_3\rbrace$ and $F_2=\lbrace e_4, e_5, e_6, e_7, e_8, e_9 \rbrace$ form a partition coarser then $\Theta^*$-partition. Hence we obtain $6$ weighted quotient trees and all of them are stars. Figure \ref{kvocienti} represents all the weighted quotient graphs with at least two vertices of weight different from $0$. Therefore, using Theorem \ref{gra_pis} we obtain
$$\widehat{W}(T) = W(T / F_1, w_{21}) + W(T / F_1, w_{31}) + W(T / F_2, w_{32}) = 6 + 24 + 30 = 60.$$

\begin{figure}[H] 
\begin{center}
\includegraphics[scale=0.6]{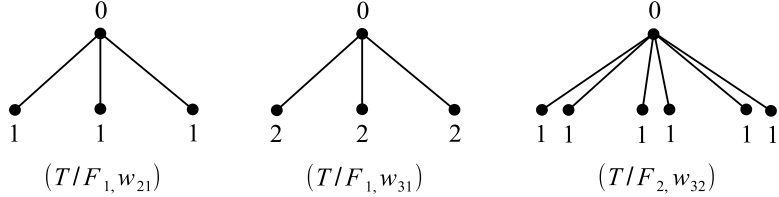}
\end{center}
\caption{\label{kvocienti} Weighted quotient graphs $(T / F_1, w_{21})$, $(T / F_1, w_{31})$, and $(T / F_2, w_{32})$.}
\end{figure}

\section{Zig-zag tubulenes}
\label{sec:zig_zag}

 Let $T$ be a zig-zag tubulene such that $c_1, c_2$ are the shortest possible cycles encircling the axis of the cylinder (see Figure \ref{zig-zag}). If $T$ has $n$ layers of hexagons, each containing exactly $h$ hexagons, then we denote it by $ZT(n,h)$. In this section we compute the Graovac-Pisanski index for $ZT(n,h)$. We always assume that $n \geq 1$ and $h \geq 2$. Moreover, let $C_1$ and $C_2$ be subgraphs of $ZT(n,h)$ induced by $c_1$ and $c_2$, respectively.

Obviously, $ZT(n,h)$ has $n+1$ layers of vertices and every layer has two types of vertices, i.e. type $0$ and type $1$. The set of vertices of type $k$ in layer $i$ is denoted by $V^k_i$. Moreover, let the vertices in $V^k_i$ be denoted as follows: $V^k_i = \lbrace v^k_{i,0}, \ldots, v^k_{i,h-1} \rbrace$. See Figure \ref{zig-zag1} for an example.

\begin{figure}[!htb]
	\centering
		\includegraphics[scale=0.8, trim=0cm 0cm 1cm 0cm]{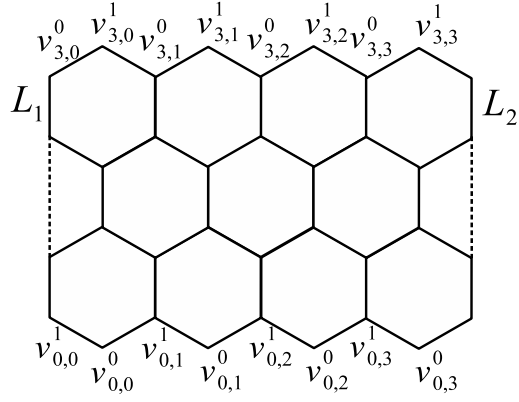}
\caption{Zig-zag tubulene $ZT(3,4)$ with vertices in $V^0_0$, $V^1_0$, $V^0_3$, and $V^1_3$.}
	\label{zig-zag1}
\end{figure}

%Let $\alpha :V(ZT(n,h)) \rightarrow V(ZT(n,h))$ be a function defined by $\alpha(v^k_{i,j}) = v^k_{i,j+1}$, where $k \in \lbrace 0,1\rbrace$, $i \in \lbrace 0, \ldots, n \rbrace$, $j \in \lbrace 0, \ldots, h-1\rbrace$ and $j+1$ is computed modulo $h$. One can easily see that $\alpha$ is an automorphism and that it is of order $h$. Intuitively, $\alpha$ is a shift in the right direction.

\subsection{The automorphisms of zig-zag tubulenes}

In this subsection we show that when $n$ is odd, the automorphism group of the graph $ZT(n,h)$ is isomorphic to the direct product of the dihedral group $D_h$ and the cyclic group $\mathbb{Z}_2$. Also, the orbits under the natural action are obtained. First, two lemmas are needed.

\begin{lemma}
\label{le1}
Let $\varphi: V(ZT(n,h)) \rightarrow V(ZT(n,h))$ be an automorphism. Then the graph induced on the vertices in the set $\varphi(V(C_1))$ is either $C_1$ or $C_2$.
\end{lemma}

\begin{proof}
The graph $ZT(n,h)$ contains exactly two disjoint cycles of length $2h$ with exactly $h$ vertices of degree 2 in the graph $ZT(n,h)$. These two are $C_1$ and $C_2$. Therefore, any automorphism maps $C_1$ to either $C_1$ or $C_2$ and the proof is complete. \qed
\end{proof}

\begin{lemma}
\label{le2}
Let $\varphi: V(C_1) \rightarrow V(C_i)$ be an isomorphism between subgraphs $C_1$ and $C_i$, where $i \in \lbrace 1,2 \rbrace$. Then there is exactly one automorphism $\overline{\varphi}: V(ZT(n,h)) \rightarrow V(ZT(n,h))$ such that $\varphi(x) = \overline{\varphi}(x)$ for any $x \in V(C_1)$.
\end{lemma}

\begin{proof}
Let $\varphi: V(C_1) \rightarrow V(C_i)$ be an isomorphism where $i \in \lbrace 1,2 \rbrace$. For any $x \in V(C_1) = V^0_0 \cup V^1_0$ we define $\overline{\varphi}(x) = \varphi(x)$. In the rest of the proof we will define function $\overline{\varphi}$ step by step such that every edge will be mapped to an edge and $\overline{\varphi}$ will be a bijection.

First let $x \in V^0_1$. Then there is exactly one $y \in V^1_0$ such that $x$ and $y$ are adjacent. Since the degree of $y$ is 3, let $y_1$ and $y_2$ be the other two neighbours of $y$ in $ZT(n,h)$. Obviously, $\overline{\varphi}(y), \overline{\varphi}(y_1)$, and $\overline{\varphi}(y_2)$ are already define and it holds that $\overline{\varphi}(y_1)$ and $\overline{\varphi}(y_2)$ are both adjacent to $\overline{\varphi}(y)$. Since the degree of $\overline{\varphi}(y)$ is 3, we define $\overline{\varphi}(x)$ to be the neighbour of $\overline{\varphi}(y)$, different from $\overline{\varphi}(y_1)$ and $\overline{\varphi}(y_2)$. This can be done for any $x \in V^0_1$.

Now let $x \in V^1_1$. Then there are exactly two vertices $y_1,y_2 \in V^0_1$ that are adjacent to $x$. It is easy to see that $\overline{\varphi}(y_1)$, $\overline{\varphi}(y_2)$ are already defined and that they have exactly one common neighbour. We define $\overline{\varphi}(x)$ to be the common neighbour of $\overline{\varphi}(y_1)$ and $\overline{\varphi}(y_2)$. This can be done for any $x \in V^1_1$. 

With the procedure above we have defined function $\overline{\varphi}$ on the set of vertices $V^0_0 \cup V^1_0 \cup V^0_1 \cup V^1_1$ such that for any two adjacent vertices $x,y \in V^0_0 \cup V^1_0 \cup V^0_1 \cup V^1_1$, it holds that $\overline{\varphi}(x)$ and $\overline{\varphi}(y)$ are also adjacent. Using induction, we can define function $\overline{\varphi}$ on the set $V(ZT(n,h))$ such that for any two adjacent vertices $x,y \in V(ZT(n,h))$ it holds that $\overline{\varphi}(x)$ and $\overline{\varphi}(y)$ are adjacent. Since $\overline{\varphi}$ is also bijective, it is an automorphism of the graph $ZT(n,h)$. It follows from the construction that $\overline{\varphi}$ is also unique. Therefore, the proof is complete. \qed
\end{proof}

\begin{theorem}
\label{glavni}
The automorphism group of the graph $ZT(n,h)$, where $n$ is odd, is isomorphic to the direct product of the dihedral group $D_h$ and the cyclic group $\mathbb{Z}_2$, i.e.
$${\aut}(ZT(n,h))\cong D_h \times \mathbb{Z}_2.$$
\end{theorem}

\begin{proof}
Lemma \ref{le1} and Lemma \ref{le2} imply that the automorphism group of the graph $ZT(n,h)$ is uniquely defined by all the isomorphisms between the subgraph $C_1$ and $C_i$, where $i \in \lbrace 1, 2 \rbrace$. The subgraph $C_1$ is a cycle of length $2h$, but any automorphism maps any vertex into a vertex of the same degree. Therefore, the automorphism group of the subgraph $C_1$ in isomorphic to the symmetric group of the $h$-gon, i.e. the dihedral group $D_h$.

For any isomorphism $\varphi: V(C_1) \rightarrow V(C_i)$ between subgraphs $C_1$ and $C_i$, $i =1,2$ we denote by $\overline{\varphi}$ uniquely defined automorphism of $ZT(n,h)$ obtained as in Lemma \ref{le2}. Moreover, let $\alpha: V(C_1) \rightarrow V(C_2)$ be a fixed isomorphism, such that $\alpha(v_{0,j}^k)=v_{n,j}^{1-k}$. To prove the theorem, we define a function $f: D_h \times \mathbb{Z}_2 \rightarrow {\aut}(ZT(n,h)) $ as follows. For any $(\beta, k) \in D_h \times \mathbb{Z}_2$, $k \in \lbrace 0, 1 \rbrace$, we define
$$f((\beta,k)) = \begin{cases}
\overline{\beta}, & \textrm{if } k=0 \\
\overline{\alpha} \circ \overline{\beta}, & \textrm{if } k = 1.
\end{cases}
$$ It is not difficult to check that $f$ is a group isomorphism. Therefore, we are done. \qed
\end{proof}

Finally, we obtain the following theorem. 
\begin{theorem}
The orbits under the natural action of the group ${\aut}(ZT(n,h))$ on   the set $V(ZT(n,h))$ are:
\begin{itemize}
\item [1.] if $n$ is odd
$$O^0_i = V^0_i \cup V^1_{n-i}, \ i \in \Big\lbrace 0, \ldots, \frac{n-1}{2} \Big\rbrace, $$
$$O^1_i = V^1_i \cup V^0_{n-i}, \ i \in \Big\lbrace 0, \ldots, \frac{n-1}{2}  \Big\rbrace. $$
\item [2.] if $n$ is even

$$O^0_i = V^0_i \cup V^1_{n-i}, \ i \in \Big\lbrace 0, \ldots, \frac{n-2}{2} \Big\rbrace, $$
$$O^1_i = V^1_i \cup V^0_{n-i}, \ i \in \Big\lbrace 0, \ldots, \frac{n-2}{2}  \Big\rbrace, $$
$$O_{\frac{n}{2}} = V^0_{\frac{n}{2}} \cup V^1_{\frac{n}{2}}.$$
\end{itemize}

\end{theorem}

\begin{proof}
It follows from the proof of Lemma \ref{le2} that for any vertex $x$ of type $k$ in layer $i$, where $i \in \lbrace 0, \ldots, n \rbrace$, $k \in \lbrace 0, 1 \rbrace$, and any vertex $y$ in layer $i$ of type $k$ or in layer $n-i$ of type $1-k$, there is an automorphism that maps $x$ to $y$. Also, if $x$ is in layer $i$ and $y$ is in layer $j$, $j \neq i, j \neq n-i$, the distance from $x$ to $C_1$ or $C_2$, i.e. $\min \lbrace d(x,C_1),d(x,C_2) \rbrace$, can not be the same as the distance from $y$ to $C_1$ or $C_2$, i.e. $\min \lbrace d(y,C_1),d(y,C_2) \rbrace$. Therefore, there is no automorphism that maps $x$ to $y$. Moreover, the vertices in $V^0_i$ can not be mapped with vertices in the set $V^1_i$ or with the vertices in $V^0_{n-i}$, except when $n$ is even and $i = \frac{n}{2}$. Therefore, the theorem follows. \qed
\end{proof}

\subsection{The Graovac-Pisanski index of zig-zag tubulenes}

In this subsection we calculate the Graovac-Pisanski index of zig-zag tubulenes. We have to consider the following four cases. The first part is explained in details, while for the remaining cases only the important results are given. We always denote by $u$ an arbitrary element of $V^0_0$ and by $v$ an arbitrary element of $V^1_0$.

\begin{enumerate}
\item $h$ is odd and $n$ is odd \\
It is enough to compute $W(O^0_0)$ and $W(O^1_0)$, since, for example, $W(O^0_1)$ of the graph $ZT(n,h)$ is exactly $W(O^0_0)$ of the graph $ZT(n-2,h)$ (the graph $ZT(n-2,h)$ is a convex subgraph of the graph $ZT(n,h)$). It is easy to see that $d(u,V^0_0) = d(v,V^1_0) = \frac{1}{2}(h^2-1)$. To compute $W(O^0_0)$ we consider two cases.
\smallskip

\begin{enumerate}
\item $h > n+2$ \\
In this case, we can draw two lines $a$ and $b$ as shown in Figure \ref{razdalje1}. 

\begin{figure}[!htb]
	\centering
		\includegraphics[scale=0.65, trim=0cm 0cm 1cm 0cm]{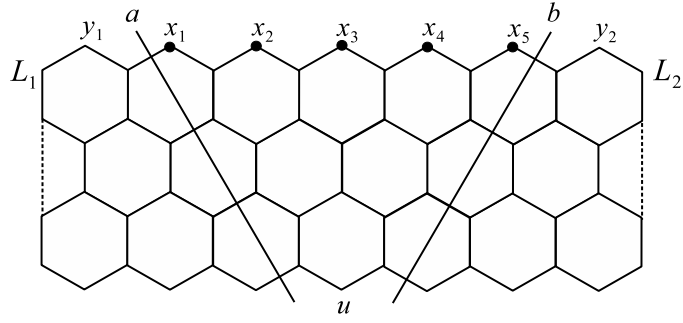}
\caption{Distances from $u$ in $ZT(3,7)$. Lines $L_1$ and $L_2$ are joined.}
	\label{razdalje1}
\end{figure}

There are $n+2$ vertices of $V^1_n$ between $a$ and $b$ (vertices $x_1, \ldots, x_5$ in Figure \ref{razdalje1}) and those vertices has distance $2n+1$ from $u$. A shortest path from $u$ to any other vertex $y$ of $V^1_n$ can be obtained by joining a shortest path from $u$ to a vertex $x \in V^1_n$ (which is between $a$ and $b$ and it is the closest one to $a$ or to $b$) and a shortest path from $x$ to $y$. For example, in Figure \ref{razdalje1}, a shortest path from $u$ to $y_2$ is composed of a shortest path form $u$ to $x_5$ and a shortest path form $x_5$ to $y_2$.
Therefore, we obtain
\begin{eqnarray*} 
d(u,V^1_n) & = & (n+2)(2n+1) + 2 \sum_{i=1}^{\frac{h-n-2}{2}}\Big (2n+1 + 2i \Big) \\
&= & \frac{1}{2}(h^2 + 2hn + n^2 + 2n).
\end{eqnarray*}
Obviously, we get
\begin{eqnarray*}
d(u,O^0_0) & = & d(u, V^0_0)  + d(u,V^1_n) \\
&= & \frac{1}{2}(2h^2 + 2hn + n^2 + 2n - 1).
\end{eqnarray*}

Finally, since any vertex in $O^0_0$ has equivalent position, one can easily deduce
\begin{eqnarray*}
W(O^0_0) & = & \frac{1}{2} \cdot |O^0_0| \cdot d(u,O^0_0) \\
 & = & \frac{h}{2}(2h^2 + 2hn + n^2 + 2n - 1).
\end{eqnarray*}

%$$d(v,V^0_n) = \frac{1}{2}(h^2 + 2hn + n^2 + 2n)$$
%$$d(v,O^0_0) = \frac{1}{2}(2h^2 + 2hn + n^2 + 2n - 1)$$
%$$W(O^0_0) = \frac{1}{2}|O^0_0|d(v,O^0_0) = \frac{1}{2}(2h)(\frac{1}{2}(2h^2 + 2hn + n^2 + 2n - 1)) = \frac{h}{2}(2h^2 + 2hn + n^2 + 2n - 1)$$

\item $h \leq n+2$ \\
In this case, all $h$ vertices of $V^1_n$ are at distance $2n+1$ from $u$ (see Figure \ref{razdalje1}). Therefore,

$$d(u,V^1_n) = h(2n+1).$$
We also obtain
$$d(u,O^0_0) =  d(u, V^0_0)  + d(u,V^1_n)=\frac{1}{2}(h^2 + 4hn + 2h - 1)$$
and
$$W(O^0_0) = \frac{h}{2}(h^2 + 4hn + 2h - 1).$$
\end{enumerate}
\bigskip

\noindent To compute $W(O^1_0)$ we also consider two cases.
\smallskip

\begin{enumerate}
\item $h > n$ \\
Similar as before, we can draw two lines $a$ and $b$ as shown in Figure \ref{razdalje2}.

\begin{figure}[h]
	\centering
		\includegraphics[scale=0.65, trim=0cm 0cm 1cm 0cm]{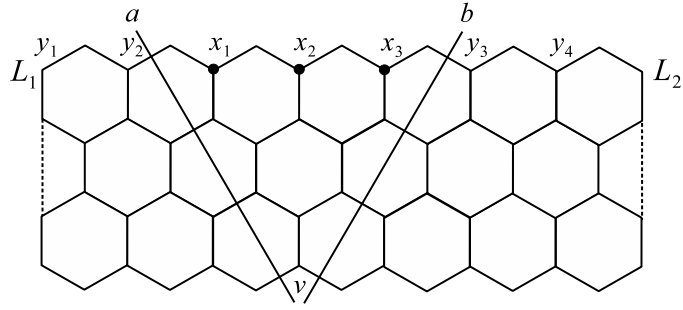}
\caption{Distances from $v$ in $ZT(3,7)$. Lines $L_1$ and $L_2$ are joined.}
	\label{razdalje2}
\end{figure}

There are $n$ vertices of $V^0_n$ between $a$ and $b$ (vertices $x_1, x_2, x_3$ in Figure \ref{razdalje2}) and those vertices has distance $2n-1$ from $v$. A shortest path from $v$ to any other vertex $y$ of $V^0_n$ can be obtained by joining a shortest path from $v$ to a vertex $x \in V^0_n$ (which is between $a$ and $b$ and it is the closest one to $a$ or to $b$) and a shortest path from $x$ to $y$. For example, in Figure \ref{razdalje2}, a shortest path from $v$ to $y_3$ is composed of a shortest path form $v$ to $x_3$ and a shortest path form $x_3$ to $y_3$.
Therefore, we obtain
\begin{eqnarray*} 
d(v,V^0_n) & = & n(2n-1) + 2 \sum_{i=1}^{\frac{h-n}{2}}\Big (2n-1 + 2i \Big) \\
& = & \frac{1}{2}(h^2 + 2hn + n^2 - 2n).
\end{eqnarray*}
Obviously, we get
\begin{eqnarray*}
d(v,O^1_0) & = & d(v, V^1_0)  + d(u,V^0_n) \\
 & = & \frac{1}{2}(2h^2 + 2hn + n^2 - 2n - 1).
\end{eqnarray*}

Finally, since any vertex in $O^1_0$ has equivalent position, one can easily deduce
\begin{eqnarray*}
W(O^1_0) & = & \frac{1}{2} \cdot |O^1_0| \cdot d(v,O^1_0) \\
 & = & \frac{h}{2}(2h^2 + 2hn + n^2 - 2n - 1).
\end{eqnarray*}

\item $h \leq n$ \\
In this case, all $h$ vertices of $V^0_n$ are at distance $2n-1$ from $v$ (see Figure \ref{razdalje2}). Therefore,

$$d(v,V^0_n) = h(2n-1).$$
We also obtain
$$d(v,O^1_0) =  d(v, V^1_0)  + d(v,V^0_n)=\frac{1}{2}(h^2 + 4hn - 2h - 1)$$
and
$$W(O^1_0) = \frac{h}{2}(h^2 + 4hn - 2h - 1).$$
\end{enumerate}

\noindent Putting all the results together, we obtain Table \ref{tabela1}.
%\begin{center}
%\begin{table}[H] \renewcommand{\arraystretch}{1.5}
%\begin{tabular}{|c||c|c|c|}
%
%
%\hline 
%& $h > n+2$ & $h=n+2$ & $h \leq n$ \\
%\hline \hline
%$d(u,V^0_0)$ & $\frac{1}{2}(h^2-1)$ & $\frac{1}{2}(h^2-1)$ & $\frac{1}{2}(h^2-1)$ \\ 
% \hline 
%$d(u,V^1_n)$ & $\frac{1}{2}(h^2 + 2hn + n^2 + 2n)$ & $h(2n+1)$ & $h(2n+1)$ \\ 
%\hline 
%$d(u,O^0_0)$ & $\frac{1}{2}(2h^2 + 2hn + n^2 + 2n - 1)$ & $\frac{1}{2}(h^2 + 4hn + 2h - 1)$ & $\frac{1}{2}(h^2 + 4hn + 2h - 1)$ \\ 
%\hline 
%$W(O^0_0)$ & $\frac{h}{2}(2h^2 + 2hn + n^2 + 2n - 1)$ & $\frac{h}{2}(h^2 + 4hn + 2h - 1)$ & $\frac{h}{2}(h^2 + 4hn + 2h - 1)$ \\ 
%\hline \hline
% $d(v,V^1_0)$ & $\frac{1}{2}(h^2-1)$ & $\frac{1}{2}(h^2-1)$ & $\frac{1}{2}(h^2-1)$  \\ 
%\hline 
%$d(v,V^0_n)$ & $\frac{1}{2}(h^2 + 2hn + n^2 - 2n)$ & $\frac{1}{2}(h^2 + 2hn + n^2 - 2n)$  & $d(v,V^0_n) = h(2n-1)$  \\ 
%\hline 
%$d(v,O^1_0)$ & $\frac{1}{2}(2h^2 + 2hn + n^2 - 2n - 1)$ & $\frac{1}{2}(2h^2 + 2hn + n^2 - 2n - 1)$ & $\frac{1}{2}(h^2 + 4hn - 2h - 1)$ \\ 
%\hline 
%$W(O^1_0)$ & $\frac{h}{2}(2h^2 + 2hn + n^2 - 2n - 1)$ & $\frac{h}{2}(2h^2 + 2hn + n^2 - 2n - 1)$ & $\frac{h}{2}(h^2 + 4hn - 2h - 1)$  \\
%\hline 
%\end{tabular} 
%\caption{\label{tabela1} Distances in $ZT(n,h)$ with $h$ odd and $n$ odd}
%\end{table}
%\end{center}

\begin{center}
\begin{table}[H] \renewcommand{\arraystretch}{1.5}
\begin{tabular}{|c||c|c|}

\hline 
& $h > n+2$ & $h \leq n+2$ \\
\hline \hline
$d(u,V^0_0)$ & $\frac{1}{2}(h^2-1)$ & $\frac{1}{2}(h^2-1)$  \\ 
 \hline 
$d(u,V^1_n)$ & $\frac{1}{2}(h^2 + 2hn + n^2 + 2n)$ & $h(2n+1)$  \\ 
\hline 
$d(u,O^0_0)$ & $\frac{1}{2}(2h^2 + 2hn + n^2 + 2n - 1)$ & $\frac{1}{2}(h^2 + 4hn + 2h - 1)$ \\ 
\hline 
$W(O^0_0)$ & $\frac{h}{2}(2h^2 + 2hn + n^2 + 2n - 1)$ & $\frac{h}{2}(h^2 + 4hn + 2h - 1)$  \\ 
\hline \hline
& $h > n$ & $h \leq n$ \\
\hline \hline
 $d(v,V^1_0)$ & $\frac{1}{2}(h^2-1)$  & $\frac{1}{2}(h^2-1)$  \\ 
\hline 
$d(v,V^0_n)$ & $\frac{1}{2}(h^2 + 2hn + n^2 - 2n)$   & $h(2n-1)$  \\ 
\hline 
$d(v,O^1_0)$ & $\frac{1}{2}(2h^2 + 2hn + n^2 - 2n - 1)$  & $\frac{1}{2}(h^2 + 4hn - 2h - 1)$ \\ 
\hline 
$W(O^1_0)$ & $\frac{h}{2}(2h^2 + 2hn + n^2 - 2n - 1)$ & $\frac{h}{2}(h^2 + 4hn - 2h - 1)$  \\
\hline 
\end{tabular} 
\caption{\label{tabela1} Distances in $ZT(n,h)$ with $h$ odd and $n$ odd.}
\end{table}
\end{center}

To compute $\widehat{W}(ZT(n,h))$, we use Formula \ref{formula}. First define the following functions from Table \ref{tabela1}.
$$\begin{array}{rcl}
f_1(n) & = & \frac{h}{2}(2h^2 + 2hn + n^2 + 2n - 1) \\
f_2(n) & = &  \frac{h}{2}(h^2 + 4hn + 2h - 1) \\
g_1(n) & = & \frac{h}{2}(2h^2 + 2hn + n^2 - 2n - 1) \\
g_2(n) & = & \frac{h}{2}(h^2 + 4hn - 2h - 1).
\end{array}$$

\noindent As already mentioned, we notice that $W(O^0_i)=f_1(n-2i)$ if $n-2i < h-2$ and $W(O^0_i)=f_2(n-2i)$ if $n-2i \geq h-2$ (and similar can be done for $W(O^1_i)$). Now consider three cases. 
\begin{itemize}
\item [(a)] $n < h-2$ \\
\noindent We obtain
$$ {W'}(ZT(n,h))  =  \sum_{i=1}^{\frac{n+1}{2}}f_1(2i-1) + \sum_{i=1}^{\frac{n+1}{2}}g_1(2i-1).$$

\item [(b)] $n = h-2$ \\
\noindent We obtain
$$ {W'}(ZT(n,h))  =  \sum_{i=1}^{\frac{n-1}{2}}f_1(2i-1) + f_2(n) + \sum_{i=1}^{\frac{n+1}{2}}g_1(2i-1).$$

\item [(c)] $n \geq h$ \\
\noindent We obtain
\begin{eqnarray*} 
{W'}(ZT(n,h))  & = & \sum_{i=1}^{\frac{h-3}{2}}f_1(2i-1) + \sum_{i=\frac{h-1}{2}}^{\frac{n+1}{2}}f_2(2i-1) \\
& + & \sum_{i=1}^{\frac{h-1}{2}}g_1(2i-1) + \sum_{i=\frac{h+1}{2}}^{\frac{n+1}{2}}g_2(2i-1).
\end{eqnarray*}

\end{itemize}
To compute all the sums from the previous cases, we use a computer program. Since $|V(ZT(n,h))| = 2h(n+1)$ and the cardinality of any orbit of $ZT(n,h)$ is $2h$, it is easy to see that $\widehat{W}(ZT(n,h)) = (n+1)W'(ZT(n,h))$. The results are presented in the first part of Table \ref{tabela5}.
\bigskip

\item $h$ is even and $n$ is odd

\noindent Important results for this case are shown in Table \ref{tabela2}. Using these results, the closed formulas for the Graovac-Pisanski index are presented in the second part of Table \ref{tabela5}. The details are similar to Case 1.

\begin{center}
\begin{table}[H] \renewcommand{\arraystretch}{1.5}
\begin{tabular}{|c||c|c|}

\hline 
& $h > n+2$ & $h \leq n+2$ \\
\hline \hline
$d(u,V^0_0)$ & $\frac{h^2}{2}$ & $\frac{h^2}{2}$  \\ 
 \hline 
$d(u,V^1_n)$ & $\frac{1}{2}(h^2 + 2hn + n^2 + 2n + 1)$ & $h(2n+1)$  \\ 
\hline 
$d(u,O^0_0)$ & $\frac{1}{2}(2h^2 + 2hn + n^2 + 2n +1)$ & $\frac{1}{2}(h^2 + 4hn + 2h)$ \\ 
\hline 
$W(O^0_0)$ & $\frac{h}{2}(2h^2 + 2hn + n^2 + 2n +1)$ & $\frac{h}{2}(h^2 + 4hn + 2h)$   \\ 
\hline \hline
& $h > n$ & $h \leq n$ \\
\hline \hline
 $d(v,V^1_0)$ & $\frac{h^2}{2}$  & $\frac{h^2}{2}$  \\ 
\hline 
$d(v,V^0_n)$ & $\frac{1}{2}(h^2 + 2hn + n^2 - 2n +1)$   & $h(2n-1)$  \\ 
\hline 
$d(v,O^1_0)$ & $\frac{1}{2}(2h^2 + 2hn + n^2 - 2n + 1)$  & $\frac{1}{2}(h^2 + 4hn - 2h)$ \\ 
\hline 
$W(O^1_0)$ & $\frac{h}{2}(2h^2 + 2hn + n^2 - 2n + 1)$ & $\frac{h}{2}(h^2 + 4hn - 2h)$  \\
\hline 
\end{tabular} 
\caption{\label{tabela2} Distances in $ZT(n,h)$ with $h$ even and $n$ odd.}
\end{table}
\end{center}

\item $h$ is odd and $n$ is even

\noindent Since the orbit $O_{\frac{n}{2}}$ is a cycle of length $2h$, it follows that $W(O_{\frac{n}{2}}) = h^3$. The other important results for this case are shown in Table \ref{tabela3}. Using these results, the closed formulas for the Graovac-Pisanski index are presented in the third part of Table \ref{tabela5}. The details are similar to Case 1.

\begin{center}
\begin{table}[H] \renewcommand{\arraystretch}{1.5}
\begin{tabular}{|c||c|c|}

\hline 
& $h > n+2$ & $h \leq n+2$ \\
\hline \hline
$d(u,V^0_0)$ & $\frac{1}{2}(h^2-1)$ & $\frac{1}{2}(h^2-1)$  \\ 
 \hline 
$d(u,V^1_n)$ & $\frac{1}{2}(h^2 + 2hn + n^2 + 2n + 1)$ & $h(2n+1)$  \\ 
\hline 
$d(u,O^0_0)$ & $\frac{1}{2}(2h^2 + 2hn + n^2 + 2n)$ & $\frac{1}{2}(h^2 + 4hn + 2h -1)$ \\ 
\hline 
$W(O^0_0)$ & $\frac{h}{2}(2h^2 + 2hn + n^2 + 2n)$ & $\frac{h}{2}(h^2 + 4hn + 2h -1)$   \\ 
\hline \hline
& $h > n$ & $h \leq n$ \\
\hline \hline
 $d(v,V^1_0)$ & $\frac{1}{2}(h^2-1)$  & $\frac{1}{2}(h^2-1)$  \\ 
\hline 
$d(v,V^0_n)$ & $\frac{1}{2}(h^2 + 2hn + n^2 - 2n +1)$   & $h(2n-1)$  \\ 
\hline 
$d(v,O^1_0)$ & $\frac{1}{2}(2h^2 + 2hn + n^2 - 2n)$  & $\frac{1}{2}(h^2 + 4hn - 2h - 1)$ \\ 
\hline 
$W(O^1_0)$ & $\frac{h}{2}(2h^2 + 2hn + n^2 - 2n)$ & $\frac{h}{2}(h^2 + 4hn - 2h -1)$  \\
\hline 
\end{tabular} 
\caption{\label{tabela3} Distances in $ZT(n,h)$ with $h$ odd and $n$ even.}
\end{table}
\end{center}

\item $h$ is even and $n$ is even \\
Since the orbit $O_{\frac{n}{2}}$ is a cycle of length $2h$, it follows that $W(O_{\frac{n}{2}}) = h^3$. The other important results for this case are shown in Table \ref{tabela4}. Using these results, the closed formulas for the Graovac-Pisanski index are presented in the last part of Table \ref{tabela5}. The details are similar to Case 1.

\begin{center}
\begin{table}[H] \renewcommand{\arraystretch}{1.5}
\begin{tabular}{|c||c|c|}

\hline 
& $h > n+2$ & $h \leq n+2$ \\
\hline \hline
$d(u,V^0_0)$ & $\frac{h^2}{2}$ & $\frac{h^2}{2}$  \\ 
 \hline 
$d(u,V^1_n)$ & $\frac{1}{2}(h^2 + 2hn + n^2 + 2n)$ & $h(2n+1)$  \\ 
\hline 
$d(u,O^0_0)$ & $\frac{1}{2}(2h^2 + 2hn + n^2 + 2n)$ & $\frac{1}{2}(h^2 + 4hn + 2h)$ \\ 
\hline 
$W(O^0_0)$ & $\frac{h}{2}(2h^2 + 2hn + n^2 + 2n)$ & $\frac{h}{2}(h^2 + 4hn + 2h)$   \\ 
\hline \hline
& $h > n$ & $h \leq n$ \\
\hline \hline
 $d(v,V^1_0)$ & $\frac{h^2}{2}$  & $\frac{h^2}{2}$  \\ 
\hline 
$d(v,V^0_n)$ & $\frac{1}{2}(h^2 + 2hn + n^2 - 2n)$   & $h(2n-1)$  \\ 
\hline 
$d(v,O^1_0)$ & $\frac{1}{2}(2h^2 + 2hn + n^2 - 2n)$  & $\frac{1}{2}(h^2 + 4hn - 2h)$ \\ 
\hline 
$W(O^1_0)$ & $\frac{h}{2}(2h^2 + 2hn + n^2 - 2n)$ & $\frac{h}{2}(h^2 + 4hn - 2h)$  \\
\hline 
\end{tabular} 
\caption{\label{tabela4} Distances in $ZT(n,h)$ with $h$ even and $n$ even.}
\end{table}
\end{center}

\end{enumerate}

Finally, the closed formulas for the Graovac-Pisanski index $\widehat{W}(ZT(n,h))$ are shown in Table \ref{tabela5}. The results for some small cases are omitted.

\begin{center}
\begin{table}[H] \renewcommand{\arraystretch}{1.5}
\begin{tabular}{|c||c|}

\hline 
$h$ odd, $n$ odd & \\
\hline 
$n < h-2$ & $\frac{h}{6}(n + 1)^2(6h^2 + 3hn + 3h + n^2 + 2n - 3)$   \\ 
 \hline 
$n=h-2, n \geq 3$ & $\frac{h}{6}(n+1)(6h^2n + 3h^2 + 3hn^2 + 12hn + 9h + n^3 - 7n - 3)$  \\ 
\hline 
$n \geq h, h \geq 5$ & $\frac{h}{6}(n+1)(h^3 + 3h^2n + 3h^2 + 6hn^2 + 12hn + 5h - 3n - 3)$ \\ 
\hline \hline
$h$ even, $n$ odd & \\
\hline 
$n < h-2$ & $\frac{h}{6}(n + 1)^2(6h^2 + 3hn + 3h + n^2 + 2n + 3)$   \\ 
 \hline 
$n=h-1, n \geq 3$ & $\frac{h}{6}(n+1)(6h^2n + 3h^2 + 3hn^2 + 12hn + 9h + n^3 - n)$  \\ 
\hline 
$n \geq h, h \geq 4$ & $\frac{h^2}{6}(n+1)(h^2 + 3hn + 3h + 6n^2 + 12n + 8)$ \\ 
\hline \hline
$h$ odd, $n$ even & \\
\hline 
$n < h-2$ & $\frac{h}{6}(n+1)(6h^2n + 6h^2 + 3hn^2 + 6hn + n^3 + 3n^2 + 2n)$   \\ 
 \hline 
$n=h-1, n \geq 4$ & $\frac{h}{6}(n+1)(6h^2n + 3h^2 + 3hn^2 + 12hn + 6h + n^3 - 4n - 3)$   \\ 
\hline 
$n \geq h, h \geq 5$ & $\frac{h}{6}(n+1)(h^3 + 3h^2n + 3h^2 + 6hn^2 + 12hn + 5h - 3n - 3)$ \\ 
\hline \hline
$h$ even, $n$ even & \\
\hline 
$n < h-2$ & $\frac{h}{6}(n+1)(6h^2n + 6h^2 + 3hn^2 + 6hn + n^3 + 3n^2 + 2n)$   \\ 
 \hline 
$n=h-2, n \geq 4$ & $\frac{h}{6}(n+1)(6h^2n + 3h^2 + 3hn^2 + 12hn + 6h + n^3 - 4n)$ \\ 
\hline 
$n \geq h, h \geq 6$ & $\frac{h^2}{6}(n+1)(h^2 + 3hn + 3h + 6n^2 + 12n + 2)$ \\ 
\hline

\end{tabular} 
\caption{\label{tabela5} Closed formulas for $\widehat{W}(ZT(n,h))$.}
\end{table}
\end{center}

\section*{Acknowledgment} 

\noindent The author Niko Tratnik was financially supported by the Slovenian Research Agency.\\
The final publication is available at Springer via http://dx.doi.org/10.1007/s10910-017-0749-5
%\begin{acknowledgements}
%\end{acknowledgements}



\end{document}